\documentclass{article}
\usepackage{graphicx} 
\usepackage{amsmath,amssymb}
\usepackage{amsthm} 

\theoremstyle{plain}    
\newtheorem{theorem}{Theorem}[section]   
\newtheorem{lemma}[theorem]{Lemma}       
\newtheorem{corollary}[theorem]{Corollary}

\title{Rigidity results for overdetermined problems of p-Laplace equation with spatially dependent data}
\author{Xi Dai\footnote{Email address:\;daixi@hnu.edu.cn}\\ School of Mathematics, Hunan University\\
Changsha 410082, Hunan Province\\
Peoples Republic of China
}
\date{}

\begin{document}

\maketitle
\begin{abstract}
Let $p>1$ and $\Omega\subset\mathbb{R}^n$ be a bounded smooth domain. We study the following overdetermined problem for the p-Laplacian:
\begin{equation*}
\left\{
\begin{array}{ll}
-\Delta_p u=f(x)&x\in\Omega,\\
u=0&x\in\partial\Omega,\\
\frac{\partial u}{\partial\nu}=g(x)&x\in\partial\Omega,
\end{array}
\right.
\end{equation*}
where $\nu$ is the inward unit normal to $\partial\Omega$ and both data may depend on position. A boundary-local version of the moving-plane argument is developed for the possibly degenerate p-Laplace operator. The advantage of this argument is that it avoids the analysis of the degenerate set of the p-Laplacian. Under some reasonable structural conditions, we prove that the solvability of the above mentioned problem forces the domain to be a ball, the source to be radial, and the Neumann datum to be a suitable constant. We also give an explicit formula for the solution. Further consequences include one directional domain symmetry,  the conclusion that full rigidity follows from reflection information in any set of directions whose linear reflections generate a dense subgroup of the orthogonal group $O(n)$, and a constant mean curvature characterization. 
\end{abstract}
\vskip 0.1in
\section{Introduction}
\par Overdetermined problems arise in many research areas, such as, in potential theory, in seeking the shape of minimizer for isoperimetric problems of domain functionals associated with partial differential equations, and also in mineral exploration. The study of overdetermined problems was initiated by J. Serrin in 1971 \cite{Se}. Serrin's rigidity theorem states that a bounded $C^{2,\alpha}$ domain $\Omega$ supports a solution to the overdetermined problem
\begin{equation*}
\left\{
    \begin{array}{ll}
        -\Delta u=1,\,&x\in \Omega  \\
         u=0,\, &x\in \partial \Omega\\
         \frac{\partial u}{\partial \nu}=c, \, &x\in \partial \Omega
    \end{array}
    \right.
\end{equation*}
must be a ball. The method serrin used to prove this theorem is the so-called moving-plane method. The moving-plane proof combines reflection, strong comparison principle, Hopf's boundary lemma, and a corner lemma. Subsequently, Weinberger \cite{We} gave an alternative proof of Serrin's rigidity theorem by a so-called p-function method. These two proof mechanisms have generated a large number of references in the literature on geometric rigidity; see, for example, \cite{BH,BHS,FZ,KP,Cr,GL,GO,GS,MP,Pr}. 

The moving plane method highly depends on facts that the solution of the problem has a definite sign and the Laplace operator is uniformly elliptic, while the p-function method essentially relies on the fact that the initial data are constants. The spatially dependent initial data change the problem in a basic way. To see this, we recall that, for any bounded smooth domain $\Omega$, the following problem
\begin{equation}\label{prob2}
\left\{
    \begin{array}{ll}
        -\Delta u=1\,&x\in \Omega,  \\
         u=0\, &x\in \partial \Omega
    \end{array}
    \right.
\end{equation}
has a unique solution $u$. Let $g(x)=\partial_{\nu} u$ on $\partial\Omega$. Then the overdetermined problem
\begin{equation}\label{prob3}
\left\{
    \begin{array}{ll}
        -\Delta u=1\,&x\in \Omega,  \\
        \frac{\partial u}{\partial\nu}=g & x\in\partial\Omega,\\
         u=0\, &x\in \partial \Omega
    \end{array}
    \right.
\end{equation}
has a unique solution, and $\Omega$ is not necessarily a ball. Therefore, one cannot expect a rigidity theorem without structural conditions on the initial data. To find those kinds of condition, we study the following model problem.
\begin{equation}\label{prob4}
\left\{    
\begin{array}{ll}        
-\Delta_p u=f(x) & x\in \Omega\\         
\frac{\partial u}{\partial \nu}=g(x) & x\in \partial\Omega\\
u(x)=0  & x\in\partial\Omega.
\end{array}    
\right.
\end{equation}
where $\Delta_p$ is a differential operator defined by $\Delta_p u=div(|\nabla u|^{p-2}\nabla u)$ and $\nu$ is the vector field of the inward unit normal to $\partial\Omega$. 

For convenience, throughout the rest of this paper, we always assume that $n\geq 2$, $p>1$, and $\Omega\subset\mathbb{R}^n$ is a bounded $C^{2,\alpha}$ domain for some $\alpha\in (0,1)$. The symbol $\nu$ denotes the inward unit normal of $\partial\Omega$. We also assume that $f\in C^\alpha(\overline{\Omega})$ and $g\in C^1(\partial\Omega)$.

Evidently, to guaranty the solvability of problem (\ref{prob4}), the divergence theorem imposes the following condition of compatibility on the initial data.
\begin{equation}\label{eq01}
\int_{\Omega}fdx=\int_{\partial\Omega}|g|^{p-2}gdS.
\end{equation}

We will give two kinds of condition that ensure the validity of Serrin type rigidity theorem for the p-Laplace operator. To express our results exactly, we introduce some concepts and assumptions.
\vskip 0.1in
{\bf Definition 1.1:} A function $u\in W_0^{1,p}(\Omega)\cup C^1(\overline{\Omega})$ is said to be a solution to the problem (\ref{prob4}) if $\frac{\partial u}{\partial\nu}=g$ pointwise on $\partial\Omega$ and satisfies the equation in the following weak sense. 
$$\int_\Omega|\nabla u|^{p-2}\nabla u\cdot\nabla\varphi dx=\int_\Omega f\varphi\,\,\,\,\,\,\,\, \forall\varphi\in C_c^\infty(\Omega).$$
\vskip 0.1in
Let $\mathbb{S}^{n-1}$ denote the set of all unit vector in $\mathbb{R}^n$ that is $\mathbb{S}^{n-1}$ is the unit sphere in $\mathbb{R}^n$, and $\Omega$ be a bounded domain in $\mathbb{R}^n$. For any $e\in\mathbb{S}^{n-1}$, and $\lambda\in\mathbb{R}$, we set
$$T_{e,\lambda}=\{x:\;x\cdot e=\lambda\},\quad\quad\mathcal{R}_{e,\lambda}(x)=x-2(x\cdot e-\lambda)e.$$
It is obvious that $T_{e,\lambda}$ is a plane orthogonal to $e$, and $\mathcal{R}_{e,\lambda}(x)$ is the reflection of $x$ about the plane $T_{e,\lambda}$. Let $\sum_{e,\lambda}=\{x\in\Omega:\;x\cdot e>\lambda\}$. We introduce the following concepts.
\vskip 0.1in
{\bf Definition 1.2:} The pair $(\sum_{e,\lambda}, T_{e,\lambda})$ is called an admissible cap of $\Omega$ if $\sum_{e,\lambda}$ is nonempty and $\mathcal{R}_{e,\lambda}(\sum_{e,\lambda})\subset\Omega$.
\vskip 0.1in
{\bf Definition 1.3:} For a direction $e\in\mathbb{S}^{n-1}$, an admissible cap $(\sum_{e,\lambda}, T_{e,\lambda})$ of $\Omega$ is said to be extremely admissible if $\overline{\mathcal{R}_{e,\lambda}(\sum_{e,\lambda})}$ is tangent to $\partial\Omega$ at some point $P$ away from $T_{e,\lambda}$, or $T_{e,\lambda}$ is orthogonal to $\partial\Omega$ at some point $Q$. 
\vskip 0.1in
We always use $(\sum_{e,\lambda_{e}}, T_{e,\lambda_{e}})$ to denote the extremely admissible cap of $\Omega$ along the direction $e$. Keeping this notation in mind, we state our assumptions on $f$ and $g$ as follows.
\vskip 0.1in
{\bf (F):}\,For any $e\in\mathbb{S}^{n-1}$, $x\in\sum_{e,\lambda_{e}}$ implies $f(x)\leq f(\mathcal{R}_{e,\lambda_{e}}(x))$.
\vskip 0.1in
{\bf (G):}\,For any $e\in\mathbb{S}^{n-1}$, $x\,\,\mbox{and}\,\,\mathcal{R}_{e,\lambda_{e}}(x)\in\partial\Omega$ implies  $g(x)\geq g(\mathcal{R}_{e,\lambda_{e}}(x))$, and $T_{e,\lambda_{e}}$ orthogonal to $\partial\Omega$ at some point $Q\in\partial\Omega$ implies $\frac{\partial g}{\partial e}(Q)\geq 0$.
\vskip 0.1in
\noindent {\bf Remark 1.1:}\,Typical examples for $g(x)$ satisfying (G) are $g\equiv constant$ and $g(x)=C(H)$ with $C$ increasing and $H$ being the mean curvature of $\partial\Omega$. Typical examples for $f(x)$ obeying (F) are $f\equiv constant$ and $f(x)=f(|x-x_0|)$ for some $x_0\in\Omega$. If $\Omega$ is a ball with center $x_0$, then $f(x)$ satisfying the Assumption (F) implies that $f(x)$ is radially symmetric, and $g(x)$ satisfying the Assumption (G) implies that $g(x)$ is a constant. 
\vskip 0.1in
At this stage, our first kind of Serrin type rigidity result can be stated as:
\vskip 0.1in
\begin{theorem}\label{re1}
Assume that $f(x)\geq 0$, $f\not\equiv 0$, and that assumptions (F)-(G) hold. The problem (\ref{prob4}) admits a solution if and only if there are $x_0\in\mathbb{R}^n$, $R>0$, and a function $F\in C^{\alpha}[0, R]$ such that
$$\Omega=B_R(x_0),\quad\quad f(x)=F(|x-x_0|),\quad\quad g\equiv c,$$
where the constant $c$ is necessarily given by
$$c=(R^{1-n}\int_0^R s^{n-1}F(s)ds)^{\frac{1}{p-1}}.$$
In this case, the unique weak solution to the problem (\ref{prob4}) can be reconstructed by
$$u(x)=\int_{|x-x_0|}^R (t^{1-n}\int_0^t s^{n-1}F(s)ds)^{\frac{1}{p-1}}dt.$$
\end{theorem}
\vskip 0.1in
Our second kind of Serrin type rigidity result is based on a mean-curvature characterization of a ball, and can be stated as:
\vskip 0.1in
\begin{theorem}\label{re2}
For any non-trivial $f(x)\in C^{\alpha}(\overline{\Omega})$, the overdetermined problem

\begin{equation}\label{prob5}
\left\{
    \begin{array}{ll}
        -\Delta_p u(x)=f(x) &x\in \Omega  \\
         u(x)=0,\quad\quad |\nabla u(x)|\neq 0 & x\in \partial \Omega\\
        |\nabla u|^{p-2}(\beta\frac{\partial u}{\partial \nu}-\frac{\partial^{2}u}{\partial \nu^{2}})=\frac{f(x)}{p-1}, \, &x\in \partial \Omega 
    \end{array}
    \right.
\end{equation}
admits a solution for some constant $\beta$ if and only if $\Omega$ is a ball.
\end{theorem}
\vskip 0.1in
If $f(x)\geq 0$ and $f(x)\not\equiv 0$, then, by strong maximum principle and Hopf's boundary lemma, any solution $u$ to the problem (\ref{prob5}) satisfies $|\nabla u(x)|\neq 0$ pointwise on $\partial\Omega$ automatically. Therefore, we have
\vskip 0.1in
\begin{corollary}\label{re3}
For any nonnegative non-trivial $f(x)\in C^{\alpha}(\overline{\Omega})$, the overdetermined problem

\begin{equation}\label{prob6}
\left\{
    \begin{array}{ll}
        -\Delta_p u(x)=f(x) &x\in \Omega  \\
         u(x)=0& x\in \partial \Omega\\
        |\nabla u|^{p-2}(\beta\frac{\partial u}{\partial \nu}-\frac{\partial^{2}u}{\partial \nu^{2}})=\frac{f(x)}{p-1}, \, &x\in \partial \Omega 
    \end{array}
    \right.
\end{equation}
admits a solution for some constant $\beta$ if and only if $\Omega$ is a ball.
\end{corollary}
\vskip 0.1in
Overdetermined problems involving $\Delta_p$ have been studied extensively in recent years. \cite{DP} studied an overdetermined problem involving $\Delta_p$ with $1<p<2$ by the moving-plane method. \cite{FGK} investigated an overdetermined problem involving a general operator with $\Delta_p$ as a special case by geometric and P-function approaches. Recently, \cite{CWZ} studied an overdetermined problem for p-Laplace equation by the continuous Steiner symmetrization method. Mention some but few. It is worth pointing out that most of the overdetermined problems involving $\Delta_p$ that have been studied so far are assumed to have spatially independent initial data. It seems that all available methods, possibly with the exception of the moving-plane method, cease to work when $f$ and $g$ are position dependent. However, the moving-plane method also cannot be used trivially in the study of degenerate overdetermined problem due to the lack of strong comparison principle. In this paper, we develop a boundary-local version of the moving-plane argument for the possibly degenerate p-Laplace operator. This boundary-local version of moving-plane argument combines reflection, strong maximum principle, weak comparison principle, strong comparison principle near the boundary, Hopf's boundary lemma, and a local version of the corner lemma. Unlike the traditional moving-plane argument used in the study of overdetermined problem for degenerate operator, our boundary-local version argument need not pay any special attention to 
the degenerate set of the differential operator. 

The rest of this paper is arranged as the following: we collect or prove some preliminary lemmas in section 2, and give a proof of our main theorem in section 3. Further consequences and remarks are included in section 4.

For simplicity, we always use the following notation in the rest of this paper.
\vskip 0.1in
\begin{itemize}
\item{$\partial_i=\frac{\partial}{\partial x_i},\quad\nabla=(\partial_1,\partial_2,\cdots,\partial_n),\quad\partial_{ij}=\frac{\partial^2}{\partial x_i\partial x_j}.$}\\
\item{For any function $u$, $u_i=\partial_i u,\quad\nabla u=(u_1,u_2,\cdots,u_n),\quad u_{ij}=\partial_{ij} u.$}\\
\item{Einstein's summation convention, which means that repeated indices indicate summation from $1$ to $n$. For example: $a_{ij}u_{ij}=\sum\limits_{i=1}^n\sum\limits_{j=1}^na_{ij}u_{ij}$, $a_{ij}\xi_i\xi_j=\sum\limits_{i=1}^n\sum\limits_{j=1}^n a_{ij}\xi_i\xi_j$, and $b_ku_k=\sum\limits_{k=1}^nb_ku_k.$}
\end{itemize}

\vskip 0.1in
\section{Preliminary Lemmas}
In this section, we will collect and prove some lemmas needed in the proof of Theorem \ref{re1}.
\vskip 0.1in
\begin{lemma}\label{le1} (strong maximum principle) \cite{PS}:\, If $u\in C^1(\Omega)$ satisfies the following inequalities
\begin{equation}\label{eq1}
\left\{
\begin{array}{ll}
-\Delta_p u(x)\geq 0 & x\in\Omega\\
u(x)\geq 0 & x\in\partial\Omega,
\end{array}
\right.
\end{equation}
then $u(x)\geq 0$ for any $x\in\Omega$, and $u(x_0)=0$ for some $x_0\in\Omega$ implies that $u(x)\equiv 0$ in $\Omega$
\end{lemma}
\vskip 0.1in
\begin{lemma}\label{le2} (weak comparison principle) \cite{SC}:\,Let $u,\, v\in C^1(\Omega)$ satisfy 
\begin{equation}\label{eq2}
\left\{
\begin{array}{ll}
-\Delta_p u(x)\geq -\Delta_p v(x) & x\in\Omega\\
u(x)\geq v(x) & x\in\partial\Omega.
\end{array}
\right.
\end{equation}
Then $u(x)\geq v(x)$ for any $x\in\Omega$.
\end{lemma}
\vskip 0.1in
In the case $p=2$, the operator $\Delta_2=\Delta$ is uniformly elliptic. Hence, the following strong comparison principle valid.
\vskip 0.1in
\begin{lemma}\label{le3} (strong comparison principle):\,Let $u,\, v\in C^1(\Omega)$ satisfy 
\begin{equation}\label{eq3}
\left\{
\begin{array}{ll}
-\Delta u(x)\geq -\Delta v(x) & x\in\Omega\\
u(x)\geq v(x) & x\in\partial\Omega.
\end{array}
\right.
\end{equation}
Then $u(x)\geq v(x)$ for any $x\in\Omega$, and $u(x_0)=v(x_0)$ for some $x_0\in\Omega$ implies $u(x)\equiv v(x)$ in $\Omega$.
\end{lemma}
\vskip 0.1in
Generally, the strong comparison principle is invalid for the operator $\Delta_p$ since it degenerate at the point $x$ where $|\nabla u|(x)=0$.
\vskip 0.1in
\begin{lemma}\label{le4}:\,Let $u,\, v\in C^2(\Omega)$ and $w(x)=u(x)-v(x)$. If we set $\Omega_N=\{x\in\Omega:\,|\nabla u|(x)=0,\;\mbox{or}\;|\nabla v|(x)=0\}$, then there exist functions $b_k(x)$ such that
$$\Delta_p u-\Delta_p v=a_{ij}(x)w_{ij}+b_k(x)w_k\;\;\;\forall x\in\Omega\backslash\Omega_N,$$
where $a_{ij}(x)=|\nabla v|^{p-2}\delta_{ij}+(p-2)|\nabla v|^{p-4} v_i v_j$.
\end{lemma}
\vskip 0.1in
{\bf Proof:}\, For any $C^2$ function $h$, by easy computation we have
$$\Delta_p h=|\nabla h|^{p-2}\Delta h+(p-2)|\nabla h|^{p-4}h_ih_jh_{ij}$$
for $C^2$ function $h$ with $|\nabla h|\neq 0$.

Setting
\begin{equation}\label{eq4}
\begin{array}{ll}
w(x)=u(x)-v(x),\\
\eta_{ij}(x)=|\nabla u|^{p-4}u_iu_ju_{ij}-|\nabla v|^{p-4}v_iv_jv_{ij},\\
d_{ij}(x)=|\nabla u|^{p-4}u_iu_j-|\nabla v|^{p-4}v_iv_j,\\
G(x)=|\nabla u|^{p-2}\Delta u-|\nabla v|^{p-2}\Delta v,\\
F_q(x)=|\nabla u|^q-|\nabla v|^q\;\;\mbox{for}\;\;q>0,
\end{array}
\end{equation}
we can easily see that
\begin{equation}\label{eq5}
\begin{array}{ll}
\eta_{ij}(x)=d_{ij}(x)u_{ij}+|\nabla v|^{p-4}v_iv_jw_{ij},\\
d_{ij}(x)=F_{p-4}(x)u_iu_j+|\nabla v|^{p-4}(u_iu_j-v_iv_j),\\
G(x)=F_{p-2}(x)\Delta u+|\nabla v|^{p-2}\Delta w.
\end{array}
\end{equation}

By the mean value theorem, there exists a function $\xi_q(x)$ between $|\nabla u|^2$ and $|\nabla v|^2$ such that
\begin{equation}\label{eq51}
F_q(x)=\frac{q}{2}\xi_q^{\frac{q-2}{2}}(x)(u+v)_kw_k.
\end{equation}
Noting $u_iu_j-v_iv_j=u_iw_j+v_jw_i$, we have
$$d_{ij}(x)=\frac{p-4}{2}\xi_q^{\frac{p-6}{2}}(x)(u+v)_kw_ku_iu_j+|\nabla v|^{p-4}(u_iw_j+v_jw_i).$$

Let 
\begin{equation}\label{eq6}
\begin{array}{ll}
C_k^1(x)=\sum\limits_{i,j=1}^n\frac{p-4}{2}\xi_{p-4}^{\frac{p-6}{2}}(x)u_iu_ju_{ij}(u+v)_k,\\
C_k^2(x)=|\nabla v|^{p-4}\sum\limits_{i=1}^n u_iu_{ik},\\
C_k^3(x)=|\nabla v|^{p-4}\sum\limits_{i=1}^n v_iu_{ki},\\
D_k(x)=C_k^1(x)+C_k^2(x)+C_k^3(x).
\end{array}
\end{equation}
Then, it is easy to see that
\begin{equation}\label{eq7}
\begin{array}{ll}
\sum\limits_{i,j=1}^n\eta_{ij}(x) &=d_{ij}(x)u_{ij}+|\nabla v|^{p-4}v_iv_jw_{ij}\\
&=|\nabla v|^{p-4}v_iv_jw_{ij}+D_k(x)w_k.
\end{array}
\end{equation}
By (\ref{eq51}), we have 
\begin{equation}\label{eq8}
\begin{array}{ll}
G(x)&=|\nabla v|^{p-2}\Delta w+F_{p-2}(x)\Delta u\\
&=|\nabla v|^{p-2}\Delta w+\frac{p-2}{2}\xi_{p-2}^{\frac{p-4}{2}}(x)\Delta u(u+v)_kw_k.
\end{array}
\end{equation}
Therefore, if we set $C_k^4(x)=\frac{p-2}{2}\xi_{p-2}^{\frac{p-4}{2}}(x)\Delta u(u+v)_k$, $b_k(x)=C_k^4(x)+(p-2)D_k(x)$, and $a_{ij}(x)=|\nabla v|^{p-2}\delta_{ij}+(p-2)|\nabla v|^{p-4}v_iv_j$, then it is easy to see that
\begin{equation}
\begin{array}{ll}
\Delta_p u-\Delta_p v &=H(x)+(p-2)\sum\limits_{i,j=1}^n\eta_{ij}\\
&=a_{ij}(x)w_{ij}+b_k(x)w_k.
\end{array}
\end{equation}
This completes the proof of Lemma\ref{le4}.
\vskip 0.1in
\begin{lemma}\label{le5}
Assume that $f(x)$ is of class $C^{\alpha}$, and $u\in C^{1}(\overline{\Omega})$ is a solution of problem
\begin{equation}
\left\{
    \begin{array}{ll}
      -\Delta_p u=f(x),\,& x\in \Omega  \\
         u=0,\, & x\in \partial \Omega\\
         |\nabla u|\neq 0,\,& x\in\partial\Omega,
    \end{array}
    \right.
\end{equation}
then, for any $x\in \partial \Omega$, we have
\begin{equation}
|\nabla u|^{p-2}[ (n-1)\frac{\partial u}{\partial \nu}H(x)-(p-1)\frac{\partial^{2}u}{\partial \nu ^{2}}]=f(x).
\end{equation}
where $H(x)$ is the mean curvature of $\partial \Omega$ at x.
\end{lemma} 
\vskip 0.1in
{\bf Proof:}\,Since $\Omega$ is of class $C^{2,\alpha}$ and $|\nabla u|\neq 0$ on $\partial\Omega$, it follows from the regularity theory \cite{TR} of $\Delta_p$ that there exists a neighbourhood $\mathcal{N}(\partial\Omega)\subset\overline{\Omega}$ of $\Omega$ such that $u\in C^{2,\alpha}(\overline{\mathcal{N}(\partial\Omega)})$. Therefore, in $\mathcal{N}(\partial\Omega)$, we have
\begin{equation}\label{eq8}
\Delta_p u=|\nabla u|^{p-2}\Delta u+(p-2)|\nabla u|^{p-4}u_iu_ju_{ij}
\end{equation} 
For any fixed $Q\in\partial\Omega$, we can choose a coordinate frame with origin at $Q$, the $x_n$ axis directed along the inward normal to $\partial\Omega$ at $Q$, and $(x_1,\cdots,x_{n-1})$ lying in the tangent plane of $\partial\Omega$ at $Q$. In this frame, we can represent $\partial\Omega$ locally by the equation
$$x_n=\varphi(x_1,x_2,\cdots,x_{n-1})$$
with $\varphi\in C^2$ and $\varphi_i(Q)=0$ for $1\leq i\leq n-1$.

By the definition of the mean curvature of a graph, we have
\begin{equation}\label{eq9}
\begin{array}{ll}
H(x) &=\frac{1}{n-1}div(\frac{\nabla\varphi}{\sqrt{1+|\nabla\varphi|^2}})\\
&=\frac{1}{n-1}\frac{\Delta\varphi}{\sqrt{1+|\nabla\varphi|^2}}+\frac{1}{n-1}\nabla(\frac{1}{\sqrt{1+|\nabla\varphi|^2}})\cdot\nabla\varphi.
\end{array}
\end{equation}
From this we get
\begin{equation}\label{eq12}
H(Q)=\frac{1}{n-1}\Delta\varphi(Q).
\end{equation}
By the assumption that $u=0$ on $\partial\Omega$, we have the following identity in a neighbourhood of point $Q$.
$$u(x_1,\cdots,x_{n-1},\varphi(x_1,\cdots,x_{n-1}))\equiv 0.$$
Diffrentiating the above identity, we obtain
\begin{equation}\label{eq13}
\begin{array}{ll}
u_i(x)+u_n(x)\varphi_i(x)=0 \\
u_{ij}(x)+u_{in}(x)\varphi_j(x)+u_{nj}(x)\varphi_{i}(x)+u_{nn}(x)\varphi_i(x)\varphi_j(x)+u_n(x)\varphi_{ij}(x)=0 
\end{array}
\end{equation}
Evaluating these equations at point $Q$, we have
\begin{equation}\label{eq14}
\begin{array}{ll}
u_i(Q)=0 &\mbox{for}\; 1\leq i\leq n-1\\
u_{ij}(Q)+u_n(Q)\varphi_{ij}(Q)=0 &\mbox{for}\; 1\leq i,j\leq n-1.
\end{array}
\end{equation}
From (\ref{eq12}) and (\ref{eq14}), we deduce
\begin{equation}\label{eq15}
\begin{array}{ll}
|u_n(Q)|=|\nabla u|(Q) \\
\sum\limits_{i=1}^{n-1}u_{ii}(Q)+(n-1)u_n(Q)H(Q)=0.
\end{array}
\end{equation}
By equation $-\Delta_p u=f(x)$ and (\ref{eq15}), we can get
$$(n-1)|u_n|^{p-2}(Q)u_n(Q)H(Q)-(p-1)|u_n|^{p-2}(Q)u_{nn}(Q)=f(Q).$$
By arbitrariness of $Q$, we finally have
\begin{equation}
|\nabla u|^{p-2}[ (n-1)\frac{\partial u}{\partial \nu}H(x)-(p-1)\frac{\partial^{2}u}{\partial \nu ^{2}}]=f(x).
\end{equation}
This completes the proof of Lemma\ref{le5}.
\vskip 0.1in
\begin{theorem}\label{re3}
The overdetermined problem
\begin{equation}\label{eq16}
\left\{
    \begin{array}{ll}
        -\Delta_p u=f(x),\,&x\in \Omega  \\
         u=0,\;\;|\nabla u|\neq 0\, &x\in \partial \Omega\\
        |\nabla u|^{p-2}(\beta\partial_{\nu} u-\partial_{\nu\nu} u)=\frac{f(x)}{p-1}, \, &x\in \partial \Omega
    \end{array}
    \right.
\end{equation}
admits a solution for some constant $\beta$ if and only if $\Omega$ is a ball.
\end{theorem}
\vskip 0.1in
{\bf Proof:}\, For sufficiency,  let $\Omega$ be a ball with radius $R$. Then the mean curvature $H(x)$ of $\partial\Omega$ is $\frac{1}{R}$. Therefore, it follows from Lemma\ref{le5} that the overdetermined problem (\ref{eq16}) with $\beta=\frac{n-1}{(p-1)R}$ has solution.

For necessity, let $u$ be a solution to the problem (\ref{eq16}). Then, from Lemma\ref{le5}, it follows that the mean curvature $H(x)$ of $\partial\Omega$ can be represented by
\begin{equation}\label{eq17}
\frac{n-1}{p-1}H(x)=(\partial_{\nu} u)^{-1}(\partial_{\nu\nu} u+\frac{f(x)}{(p-1)|\nabla u|^{p-2}}).
\end{equation}
By the boundary condition
$$|\nabla u|^{p-2}(\beta\partial_{\nu} u-\partial_{\nu\nu} u)=\frac{f(x)}{p-1},$$
we can deduce that
\begin{equation}\label{eq18}
(\partial_{\nu} u)^{-1}(\partial_{\nu\nu} u+\frac{f(x)}{(p-1)|\nabla u|^{p-2}})=\beta.
\end{equation}
Combining (\ref{eq17}) and (\ref{eq18}), we get
$$H(x)=\frac{p-1}{n-1}\alpha=constant.$$
Therefore, it follows from Alexandrov's characterization of spheres \cite{Al, Ale} that $\Omega$ is a ball.
\vskip 0.1in
In the case $0\leq f(x)\not\equiv 0$, the condition $|\nabla u|\neq 0$ on $\partial\Omega$ can be removed. To see this, we need the following Hopf's boundary point lemma.
\vskip 0.1in
\begin{lemma}\label{le6} \cite{SC}:
If $u(x)\in C^{1}$ satisfies the following inequality in weak sense
\begin{equation}\label{eq19}
\left\{
    \begin{array}{ll}
        -\Delta_p u\geq 0,\,& x\in \Omega  \\
         u>0, & x\in\Omega.
    \end{array}
    \right.
\end{equation}
Then, for any $x_0\in\partial\Omega$, $u(x_0)=0$ implies $\partial_{\nu} u(x_0)>0$.
\end{lemma}
\vskip 0.1in
\begin{corollary}\label{re4}
Assume $f(x)\geq 0$, and $f(x)\not\equiv 0$. Then the overdetermined problem
\begin{equation}\label{eq20}
\left\{
    \begin{array}{ll}
        -\Delta_p u=f(x),\,&x\in \Omega  \\
         u=0,\;\; & x\in \partial \Omega\\
        |\nabla u|^{p-2}(\beta\partial_{\nu} u-\partial_{\nu\nu}u)=\frac{f(x)}{p-1}, \, &x\in \partial \Omega
    \end{array}
    \right.
\end{equation}
admits a solution for some constant $\beta$ if and only if $\Omega$ is a ball.
\end{corollary}
\vskip 0.1in
{\bf Proof:}\, Let $u\in C^1$ be a solution to the problem
\begin{equation}
\left\{
\begin{array}{ll}
-\Delta_p u(x)=f(x), & x\in\Omega,\\
u(x)=0,& x\in\partial\Omega.
\end{array}
\right.
\end{equation}
By the assumption $0\leq f(x)\not\equiv 0$, it follows from Lemma \ref{le1} that $u(x)$ satisfies
\begin{equation}
\left\{
\begin{array}{ll}
-\Delta_p u(x)\geq 0, & x\in\Omega,\\
u(x)>0, & x\in\Omega,\\
u(x)=0,& x\in\partial\Omega.
\end{array}
\right.
\end{equation}
From this and Lemma \ref{le6}, we conclude that $|\nabla u|\neq 0$ for any $x\in\partial\Omega$. Now, the conclusion of Corollary \ref{re4} follows from Theorem \ref{re3}.
\vskip 0.1in
Let $\Omega$ be a bounded domain of class $C^2$ and $Q\in\partial\Omega$. A ball $B_R$ is called a inner ball of $\Omega$ at $Q$ if it is internally tangent to $\partial\Omega$ at $Q$ and $B_R\subset\Omega$. Obviously, inner ball of $\Omega$ at point $Q$ is not unique. The quantity 
$$R_{Q}^{in}=\sup\{R\;||\;B_R\,\,\mbox{is a inner ball of}\,\,\Omega\,\,\mbox{at}\,\,Q\}$$
is called the inner radius of $\Omega$ at $Q$.

Let $B_R(Q)$ be the ball with center at $Q$ and radius $R$. The following corner lemma is a modified version of lemma 2 in \cite{Se}.
\vskip 0.1in
\begin{lemma}\label{le7}\cite{Se}: 
Let $D^*$ be a domain with $C^2$ boundary and $T$ be a plane containing the normal to $\partial D^*$ at some point $Q$. Let $D$ denote the portion of $D^*$ lying on some particular side of $T$, Let $0<R\leq R_Q^{in}$ and $D_R=D\cap B_R(Q)$. Suppose $a_{ij}(x)$ and $b_i(x)$ are uniformly bounded functions which satisfy
$$a_{ij}(x)\xi_i\xi_j\geq\lambda |\xi|^2\;\;\;\mbox{with}\;\;\;\lambda=\mbox{positive constant}$$
and 
$$|a_{ij}(x)\xi_i\eta_j|\leq K(|\xi\cdot\eta|+|\xi|\cdot|d|)\;\;\;\mbox{for some costant}\;\;\; K.$$
Where $\xi\in\mathbb{R}^n$, $\eta$ is the unit vector normal to the plane $T$ and $d$ is the distance from $T$.

If $w\in C^2(\overline{D_R})$ satisfies
\begin{equation}
\left\{
\begin{array}{ll}
Lw(x)=a_{ij}(x)w_{ij}+b_i(x)w_i\leq 0 & x\in D_R\\
w(x)\geq 0 & x\in D_R\\
w(Q)=0,
\end{array}
\right.
\end{equation}
then, for any direction $s$ at $Q$ which enters $D_R$ non-tangentially, we have
$$\partial_{s} w(Q)>0\,\,\,\;\mbox{or}\,\,\,\;\partial_{ss}w(Q)>0$$
unless $w\equiv 0$.
\end{lemma}
\section{Proof of Theorem 1.1}
Before proceeding our proof, we prove the following lemma first.
\vskip 0.1in
\begin{lemma}\label{le41} Assume that $f\geq 0$, $f\not\equiv 0$, and that (F)-(G) hold. If the overdetermined problem
\begin{equation}\label{eq41}
\left\{
\begin{array}{ll}
-\Delta_p u(x)=f(x)& x\in\Omega\\
\partial_{\nu}u(x)=g(x)& x\in\partial\Omega\\
u(x)=0& x\in\partial\Omega
\end{array}
\right.
\end{equation}
has a solution, then $\Omega$ must be a ball.
\end{lemma}
\vskip 0.1in
{\bf Proof:}\;Let $u(x)$ be the solution to problem (\ref{eq41}). By the strong maximum principle Lemma \ref{le1}, we have $u(x)>0$ for any $x\in\Omega$. By the regularity theory, we also know that $u\in C^{1,\alpha}(\overline{\Omega})$ (see \cite{TR}). Since $u(x)=0$ for $x\in\partial\Omega$, it follows from Lemma \ref{le6} that $g(x)=\partial_{\nu} u(x)>0$ for any $x\in\partial\Omega$. Because $\partial\Omega$ is compact, we have $\min\limits_{x\in\partial\Omega}g(x)=m>0$. By the continuity of $\nabla u$, we know that there exists a neighbourhood $\mathcal{N}(\partial\Omega)\subset\Omega$ of $\partial\Omega$ such that
$$|\nabla u|(x)\geq\frac{m}{2}\;\;\; \mbox{for any}\;\;\; x\in\overline{\mathcal{N}(\partial\Omega)}.$$
From this, we see that the operator $div(|\nabla u|^{p-2}\nabla\cdot)$ is uniformly elliptic  on $\mathcal{N}(\partial\Omega)$. Now, by the regularity theory of uniformly elliptic equations, we have $u(x)\in C^2(\overline{\mathcal{N}(\partial\Omega)})$, and $\Delta_p u$ can be rewritten as
\begin{equation}\label{eq31}
\Delta_p u=|\nabla u|^{p-2}\Delta u+(p-2)|\nabla u|^{p-4}u_iu_ju_{ij}\;\;\;\forall x\in \mathcal{N}(\partial\Omega).
\end{equation}

From now on, we begin our procedure of moving plane. For an arbitrarily fixed $e\in\mathbb{S}^{n-1}$, and any $\lambda\in\mathbb{R}$, we set
$$T_{e,\lambda}=\{x\in\mathbb{R}^n:\;x\cdot e=\lambda\}\quad\quad\mathcal{R}_{e,\lambda}(x)=x-2(x\cdot e-\lambda)e.$$
Let $\sum_{e,\lambda}=\{x\in\Omega:\;x\cdot e>\lambda\}$ and $\mathcal{R}_{e,\lambda}(\sum_{e,\lambda})=\{\mathcal{R}_{e,\lambda}(x):\;x\in\sum_{e,\lambda}\}$. Since $\Omega$ is bounded, it is evident that $\sum_{e,\lambda}=\phi$ for $\lambda>0$ large enough. Set $\lambda_0=\inf\{\lambda:\;\sum_{e,\lambda}=\phi\}$. Because $\Omega$ is of class $C^{2,\alpha}$, it is easy to see that $(\sum_{e,\lambda}, T_{e,\lambda})$ is an admissible cap of $\Omega$ for $\lambda<\lambda_0$ and sufficiently close to $\lambda_0$. Moving the plane $T_{e,\lambda}$ by decreasing $\lambda$ continuously and keeping the admissibility of the cap $(\sum_{e,\lambda}, T_{e,\lambda})$, we can finally arrive an extreme position with $\lambda=\lambda_e$ such that $(\sum_{e,\lambda_e}, T_{e,\lambda_e})$ is an extremely admissible cap of $\Omega$. For simplicity of notation, we denote $\sum_{e,\lambda_e}$ by $\sum_{e}$, $T_{e,\lambda_e}$ by $T_e$, $\mathcal{R}_{e,\lambda_e}$ by $\mathcal{R}_e$, and $\mathcal{R}_{e,\lambda_e}(\sum_{e,\lambda_e})$ by $\sum_e^{\mathcal{R}_e}$. Let $\sum_e^c=\Omega\backslash{\sum_e}$. We aim to prove $\sum_e^c=\sum_e^{\mathcal{R}_e}$. If this is done, then $\Omega$ is symmetry about $T_e$ and the conclusion of Lemma \ref{le41} follows immediately from the arbitrariness of $e\in\mathbb{S}^{n-1}$.

By the definition of extremely admissible cap of $\Omega$, we have following two possible cases: 
\vskip 0.1in
(i)\; $\partial\sum_e\subset\overline{\Omega}$ and it tangent to $\partial\Omega$ at some point $P$ away from $T_e$, or
\vskip 0.1in
(ii)\; $T_e$ orthogonal to $\partial\Omega$ at some point $Q$.

To prove $\sum_e^c=\sum^{\mathcal{R}_e}_e$, we introduce a new function $v$ defined in $\sum_e$ by the formula
$v(x)=u(\mathcal{R}_e(x))$ for any $x\in\overline{\sum_e}$. It is easy to see that $v(x)$ satisfies
\begin{equation}\label{eq32} 
\left\{
\begin{array}{ll}
-\Delta_p v(x)=f(\mathcal{R}_e(x)) & x\in\sum_e\\
v(x)=u(x) & x\in T_e\\
v(x)\geq 0 & x\in\partial\sum_e\backslash T_e.
\end{array}
\right.
\end{equation}
Since $(\sum_e, T_e)$ is extremely admissible, by (F), we have 
\begin{equation*}
\begin{array}{cc}
f(x)\leq f(\mathcal{R}_e(x))\;\;\; \forall x\in\sum_e
\end{array}
\end{equation*}.
Noticing that $u(x)$ satisfies
\begin{equation}\label{eq33} 
\left\{
\begin{array}{ll}
-\Delta_p u(x)=f(x) & x\in\sum_e\\
u(x)=0 & x\in\partial\sum_e\backslash T_e,
\end{array}
\right.
\end{equation}
we get
\begin{equation}\label{eq34} 
\left\{
\begin{array}{ll}
-\Delta_p v(x)\geq -\Delta_p u(x) & x\in\sum_e\\
v(x)-u(x)=0 & x\in T_e\\
v(x)-u(x)\geq 0 & x\in\partial\sum_e\backslash T_e.
\end{array}
\right.
\end{equation}

At this stage, if the operator $\Delta_p$ obeys strong comparison principle, then, by an argument similar to that used by Serrin in \cite{Se}, we can reach our final goal. Sadly, the strong comparison principle is, in general, not valid for $\Delta_p$ due to the fact that it may be degenerate when $p>2$. Fortunately, it is not too bad that the strong comparison principle is still valid for $\Delta_p$ near the boundary of $\Omega$. Therefore, we can still reach our final goal by some careful analysis near the boundary of $\Omega$.

In case (i), we have $\partial\sum_e^c\cap(\partial\sum^{\mathcal{R}_e}_e\backslash{T_e})\neq\phi$. if $\sum_e^c\neq\sum^{\mathcal{R}_e}_e$, then we can see that
\begin{equation*}
\begin{array}{cc}
\sum_e^c\cap(\partial\sum^{\mathcal{R}_e}_e\backslash{T_e})\neq\phi.
\end{array}
\end{equation*}
Therefore, we can choose a point $P\in (\partial\sum_e^c\cap(\partial\sum^{\mathcal{R}_e}_e\backslash T_e)$ such that $\partial\sum^{\mathcal{R}_e}_e$ tangent to $\partial\Omega$ at $P$ and $\sum_e^c\cap B_r(P)\cap(\partial\sum^{\mathcal{R}_e}_e\backslash T_e)\neq\phi$ for arbitrary ball $B_r(P)$ centered at $P$. Let $\mathcal{V}_r^{\mathcal{R}_e}(P)=\sum_e^c\cap B_r(P)\cap\sum^{\mathcal{R}_e}_e$, and $\partial_p\mathcal{V}_r^{\mathcal{R}_e}(P)=\sum_e^c\cap B_r(P)\cap(\partial\sum^{\mathcal{R}_e}_e\backslash T_e)$ which we call it the principle boundary of $\partial\mathcal{V}_r^{\mathcal{R}_e}(P)$. Let $Q$ be the reflection of $P$ about $T_e$, $\mathcal{V}_r(Q)$ be the reflection of $\mathcal{V}_r^{\mathcal{R}_e}(P)$ about $T_e$, and $\partial_p\mathcal{V}_r(Q)$ be the reflection of $\partial_p\mathcal{V}_r^{\mathcal{R}_e}(P)$. Evidently, $Q\in\partial\mathcal{V}_r(Q)$, $\mathcal{V}_r(Q)\subset\sum_e$, and $\partial_p\mathcal{V}_r(Q)\subset\partial\Omega$. Choosing $r_0>0$ so small that $\mathcal{V}_{r_0}(Q)\;\;\mbox{and}\;\;\mathcal{V}^{\mathcal{R}_e}_{r_0}(P)\subset\mathcal{N}(\partial\Omega)$, we have $u(x),v(x)\in C^2(\overline{\mathcal{V}_{r_0}(Q))}$. 

Let $w(x)=v(x)-u(x)$. We introduce a linear differential operator $L$ such that
$$Lw(x)=\Delta_p v(x)-\Delta_p u(x).$$

By Lemma \ref{le4}, $Lw(x)$ can be written as
$$Lw(x)=a_{ij}(x)w_{ij}+b_k(x)w_k\;\;\;\forall x\in\mathcal{V}_{r_0}(Q)$$
with $a_{ij}(x)=|\nabla u|^{p-2}(x)\delta_{ij}+(p-2)|\nabla u|^{p-4}(x)u_i(x)u_j(x)$ and $b_i(x)$ being uniformly bounded functions.

By the computations in the proof of Lemma \ref{le5}, we can easily see that
$$
\begin{array}{ll}
a_{ij}(Q)\xi_i\xi_j &=|\nabla u|^{p-2}(Q)\delta_{ij}\xi_i\xi_j+(p-2)|\nabla u|^{p-4}(Q)u_i(Q)u_j(Q)\xi_i\xi_j\\
&\geq\min\{1,p-1\}|\partial_{\nu} u|^{p-2}(Q)|\xi|^2\geq\kappa|\xi|^2
\end{array}
$$
with $\kappa$ being a positive constant. By the continuity of $a_{ij}(x)$, we can assume that the matrix $(a_{ij}(x)$ is uniformly positive definite in $\mathcal{V}_{r_0}(Q)$ by choosing $r_0$ even more small. Therefore, $L$ is uniformly elliptic in $\mathcal{V}_{r_0}(Q)$. By weak comparison principle Lemma \ref{le2} and (\ref{eq34}), we have $w(x)\geq 0$ on $\overline{\sum_e}$. Noticing further that $\partial_p\mathcal{V}^{\mathcal{R}_e}_{r_0}(P)\neq\phi$, we have $w(x)>0$ on $\partial_p\mathcal{V}_{r_0}(Q)$. Consequently, $w(x)$ satisfies 
\begin{equation}\label{eq35} 
\left\{
\begin{array}{ll}
Lw(x)\leq 0 & x\in\mathcal{V}_{r_0}(Q)\\
w(x)\geq 0 & x\in\overline{\mathcal{V}_{r_0}(Q)}\\
w(x)>0 & \mbox{on}\;\;\partial_p\mathcal{V}_{r_0}(Q).
\end{array}
\right.
\end{equation}
From this and the strong maximum principle for uniformly elliptic equations, we get
$$w(x)>0\;\;\; x\in\mathcal{V}_{r_0}(Q).$$
Since $w(Q)=0$ and $\mathcal{V}_{r_0}(Q)$ verify the inner ball condition, it follows from Hopf's boundary point lemma that
\begin{equation}\label{eq36}
\frac{\partial w}{\partial\nu}(Q)>0.
\end{equation}
On the other hand, by the assumption of $g(x)$, we have
$$\frac{\partial v}{\partial\nu}=g(P)\leq g(Q)=\frac{\partial u}{\partial\nu}(Q).$$
This contradicts to (\ref{eq36}). Therefore, we have $\sum_e^c=\sum^{\mathcal{R}_e}_e$ in case (i).

\vskip 0.1in
It remains to treat case (ii). To finish this task, we choose a rectangular coordinate frame with $Q=0$, $T_e=\{x\in\mathbb{R}^n:\;x_1=0\}$, $e=e_1=(1,0,\cdots,0)$, and $\nu(Q)=e_n=(0,\cdots,0,1)$. In this frame, $\sum_e$ lies in the half space $x_1>0$, and
\begin{equation}\label{eq37}
v(x_1,x_2,\cdots,x_n)=u(-x_1,x_2,\cdots,x_n)
\end{equation}
due to $v(x)$ being the reflection of $u(x)$. Moreover, $\partial\Omega$ can be represented by 
$$x_n=\phi(x')\quad\quad x'=(x_1,x_2,\cdots,x_{n-1})\in\mathbb{R}^{n-1}.$$
By Lemma \ref{le5}, we have
\begin{equation}\label{eq38}
\left\{
\begin{array}{ll}
u_i(0)=0\,\,\,\,\;u_n(0)=g(0)\\
u_{ij}(0)+g(0)\phi_{ij}(0)=0\\ 
u_{nn}(0)=\frac{1}{p-1}((n-1)g(0)H(0)-f(0)g^{2-p}(0))
\end{array}
\right.
\end{equation}
with $i,j$ ranging from $1$ to $n-1$, and $H(0)$ being the mean curvature of $\partial\Omega$ at $0$.
From (\ref{eq37}) and (\ref{eq38}), we conclude that all first and second derivatives of $u$ and $v$ agree at $0$, possibly with the exception of $u_{1j}$ and $v_{1j}$ for $j$ ranging from $2$ to $n$. Next, we prove $u_{1j}(0)=v_{1j}(0)=0$ for $j=2,\cdots, n$ under the assumption that $\sum_e^c\neq\sum^{\mathcal{R}_e}_e$. 

At first, it follows from the second equation in (\ref{eq38}) that 
\begin{equation}\label{eq39}
u_{1j}(0)+g(0)\phi_{1j}(0)=0,\,\,\,\mbox{for}\,\,\, j=2,\cdots,n-1.
\end{equation}
In the second, if $\sum^{\mathcal{R}_e}_e\neq\sum_e^c$, we may assume $\partial\sum^{\mathcal{R}_e}_e\backslash T_e\subset\sum_e^c$. Otherwise, we return to the case (i).

Taking the assumption $\partial\sum^{\mathcal{R}_e}_e\backslash T_e\subset\sum_e^c$ into account, it follows easily from a Taylor's expansion formula for $\phi(x)$ at $0$ that
$$\phi_{1l}(0)=0,\,\,\,\mbox{for}\,\,\,l=2,\cdots,n-1.$$
Combining this with (\ref{eq39}), we obtain
$$v_{1l}(0)=-u_{1l}(0)=0,\,\,\,\mbox{for}\,\,\,l=2,\cdots,n-1.$$

Lastly, to prove $u_{1n}(0)=v_{1n}(0)=0$, we choose a ball $B_r(0)$ with center at $0$ and radius $r$ so small that $D^*_r=B_r(0)\cap\Omega\subset\mathcal{N}(\partial\Omega)$. Let $w(x)=v(x)-u(x)$. Then, by Lemma \ref{le4}, we may assume that $Lw(x)=\Delta_p v(x)-\Delta_p u(x)$ defines a linear uniformly elliptic operator 
$$L=a_{ij}(x)\partial_{ij}+c_k\partial_k$$ 
in $D_r=D^*_r\cap\sum_e$ by choosing $r$ even more small. Evidently $w(x)\geq 0$ on $D_r$ and $w(x)>0$ on $B_r(0)\cap(\partial\sum_e\backslash T_e)$. Furthermore, $Lw\leq 0$ due to the fact that $(\sum_e, T_e)$ is an extremely admissible cap of $\Omega$. Therefore, it follows from the strong maximum principle that $w(x)>0$ on $D_r$. Noting that the only possible nonzero term of the first and second derivatives of $w$ at $0$ is $w_{1n}(0)$, then, for $x\in D_r$, $w(x)$ can be expanded at $Q$ as
\begin{equation}\label{eq3100}
w(x)=w_{1n}(0)x_1x_n+o(|x|^2).
\end{equation}
By the boundary condition $u_\nu(x)=g(x)$ and the local expression of $\partial\Omega$, we have  
\begin{equation}\label{eq310}
u_n(x)-\sum_{i=1}^{n-1}u_{i}(x)\phi_{i}(x')\equiv g(x',\phi(x'))\sqrt{1+|\nabla\phi(x')|^2}
\end{equation}
in a neighbourhood of $0$.

Differentiating (\ref{eq310}) with $x_1$ and evaluating at $0$, we get
\begin{equation}\label{eq311}
u_{1n}(0)=\partial_{x_1} g(0).
\end{equation}
Hence
$$w_{1n}(0)=v_{1n}(0)-u_{1n}(0)=-2\partial_{x_1} g(0).$$
By the assumption (G), we have $\partial_{x_1} g(0)\geq 0$. This leads to $w_{1n}(0)\leq 0$. Evidently, $x_1>0$, $x_n>0$, and $w(x)>0$ for $x\in D_r$. These and (\ref{eq3100}) compel $w_{1n}(0)=0$.

At this stage, Lemma \ref{le7} can be used to reach a contradiction. However, in order to use Lemma \ref{le7}, one must verify the following two construction conditions for the coefficient matrix $(a_{ij}(x))_{n\times n}$ of linear differential operator $L$. 
\vskip 0.1in
(C1):\,There exists a positive constant $\kappa$ such that 
$$a_{ij}(x)\xi_i\xi_j\geq\kappa |\xi|^2\,\,\,\mbox{for any}\,\,\, \xi\in R^n$$
\vskip 0.1in
(C2):\,There exists a positive constant $K$ such that 
$$|a_{ij}(x)\xi_i\eta_j|\leq K(|\xi\cdot\eta|+|\xi|d)$$
where $\xi$ is an arbitrary real vector in $R^n$, $\eta$ is the unit normal to the plane $T_e$, and $d$ is the distance from $T_e$.

As shown in the previous paragraph of this section, condition (C1) can be verified by making use of the formula given in Lemma \ref{le4}. Unfortunately, this formula is not the right form for verifying condition (C2). To verify condition (C2), we rewrite the linear operator $L$ as symmetric form. To this end, we set $a(t)=t^{p-2}$ and $h(t)=(p-2)t^{p-4}$. It is easy to see that
$$\Delta_p u=a(|\nabla u|)\Delta u+h(|\nabla u|)u_iu_ju_{ij}\,\,\,\mbox{in}\,\,\, D_r.$$
Let $w=v-u$, $b(x)=u(x)+v(x)$. Then, by simple computation, we have
\begin{equation}\label{eq312}
\begin{array}{ll}
2(\Delta_p v-\Delta_p u) &=(a(|\nabla v|)+a(|\nabla u|))\Delta w+(h(|\nabla v|)v_iv_j+h(|\nabla u|)u_iu_j)w_{ij}\\
&+(a(|\nabla v|)-a(|\nabla u|))\Delta b+(h(|\nabla v|)v_iv_j-h(|\nabla u|)u_iu_j) b_{ij}
\end{array}
\end{equation}
Let
$$2R(x)=(a(|\nabla v|)-a(|\nabla u|))\Delta b+(h(|\nabla v|)v_iv_j-h(|\nabla u|)u_iu_j) b_{ij}.$$
As in the proof of Lemma \ref{le4}, $R(x)$ can be rewritten as
$$R(x)=\sum\limits_{k=1}^nc_k(x)w_k$$
with $c_k(x)$ being some uniformly bounded functions.

Let 
$$2a_{ij}(x)=(a(|\nabla v|)+a(|\nabla u|))\delta_{ij}+(h(|\nabla v|)v_iv_j+h(|\nabla u|)u_iu_j).$$
Evaluating $a_{ij}(x)$ at $0$, we have
$$a_{ij}(0)=\left\{\begin{array}{ll}
0&i\neq j\\
g^{p-2}(0)&i=j=1,2,\cdots,n-1\\
(p-1)g^{p-2}(0)&i=j=n
\end{array}\right.$$
Since $g(0)>0$ and $p>1$, we know that the linear operator 
$$L=a_{ij}(x)\partial_{ij}+c_k(x)\partial_k$$
determined by $Lw=\Delta_p v-\Delta_p u$ is uniformly elliptic at the point $0$. By the continuity of $a_{ij}(x)$, we conclude that $L$ is uniformly elliptic in $D_r$ for $r$ small enough. This verifies condition (C1).

To verify condition (C2), we denote by $\xi$ an arbitrary vector in $R^n$ and $\eta$ the unit normal to plane $T_e$. Let $p=\nabla u=(u_1,u_2,\cdots,u_n)$, $q=\nabla v=(v_1,v_2,\cdots,v_n)$, and  $B=h(|q|)(q\cdot\xi)(q\cdot\eta)+h(|p|)(p\cdot\xi)(p\cdot\eta)$. By simple computation, we have
$$2a_{ij}\xi_i\eta_j=(a(|\nabla v|)+a(|\nabla u|))\delta_{ij}\xi_i\eta_j+B.$$
Since $v$ is the reflection of $u$ in the plane $T_e$, in our rectangular coordinate frame, we have
$$v(x_1,x_2,\cdots,x_n)=u(-x_1,x_2,\cdots,x_n).$$
This leads to
$$u=v,\,\,p=(u_1,u_2,\cdot,u_n),\,\,q=(-u_1,u_2,\cdots,u_n)\,\,\,\,\,\,\mbox{on}\,\,T_e$$
Thus, on $T_e$, there holds
$$|p|=|q|,\quad a(|p|)=a(|q|),\quad h(|p|)=h(|q|).$$
Since $\eta=(1,0,\cdots,0)$, we have
\begin{equation*}
\begin{array}{ll}
B &=h(|q|)(q\cdot\xi)(q\cdot\eta)+h(|p|)(p\cdot\xi)(p\cdot\eta)\\
 &=h(|p|)u_1(p-q)\cdot\xi=2h(|p|)u_1^2\xi_1\\
 &=2h(|p|)u_1^2\xi\cdot\eta \quad\quad\quad\quad\quad\quad \mbox{on}\,\,T_e.
\end{array}
\end{equation*}
By continuity, it then follows that off $T_e$ 
$$|B|\leq M(|\xi\cdot\eta|+|\xi|x_1)$$
for some positive constant $M$. Noting that also $a(|p|)+a(|q|)\leq M_1$ for some positive constant $M_1$, we have
$$|a_{ij}(x)\xi_i\eta_j|\leq(M+M_1)(|\xi\cdot\eta|+|\xi|\cdot x_1).$$
This verifies condition (C2).

Since $(\sum_e, T_e)$ is an extremely admissible cap of $\Omega$, we can see from the assumption (F) that $w(x)$ satisfies
\begin{equation*}
\begin{array}{ll}
Lw=a_{ij}(x)w_{ij}+c_k(x)w_k\leq 0 & x\in D_r\\
w(x)>0 & x\in D_r\\
w(0)=0
\end{array}
\end{equation*}
It follows from Lemma \ref{le7} that, for any direction $s$ at $Q$ which enters $D_r$ non-tangentially, there holds
$$\frac{\partial w}{\partial s}(0)>0\quad\mbox{or}\quad\frac{\partial^2w}{\partial s^2}(0)>0.$$
This is a contradiction, since all the first and second partial derivatives of $w$ equal zero at $0$.

Therefore, we have $\sum_e^c=\sum^{\mathcal{R}_e}_e$. Consequently, $\Omega$ is a ball due to the arbitrariness of $e\in\mathbb{S}^{n-1}$. Once $\Omega$ is a ball, it follows immediately from the assumptions (F) and (G) that $f$ and $g$ are radially symmetric. This completes the proof of Lemma \ref{le41}.
\vskip 0.1in
{\bf Proof of Theorem \ref{re1}}\; Assume first that problem \ref{prob4} has a solution. Then, by Lemma \ref{le41}, there are $x_0\in\Omega$, positive number $R$, nonnegative function $F\in C^\alpha[0, R]$, and real number $c$ such that
$$\Omega=B_R(x_0),\quad\quad f(x)=F(|x-x_0|),\quad\quad g\equiv c.$$
Moreover, it follows from the divergence theorem that $c$ must be given by the formula
$$c=(R^{1-n}\int_0^Rs^{n-1}F(s)ds)^{1/(p-1)}.$$

Conversely, suppose $\Omega=B_R(x_0)$, $f(x)=F(|x-x_0|)$, and $c$ is given by the above formula. Define
\begin{equation*}
q(r)=\left\{
\begin{array}{ll}
(r^{1-n}\int_0^r s^{n-1}F(s)ds)^{1/(p-1)}&r>0\\
0&r=0
\end{array}
\right.
\end{equation*}
and set
$$U(r)=\int_r^R q(t)dt,\quad\quad u(x)=U(|x-x_0|).$$
Then $U'(r)=-q(r)$ and $(r^{n-1}|U'|^{p-2}U')=-r^{n-1}F(r)$. Thus $-\Delta_p u=F(|x-x_0|)$ weakly in the ball and $u=0$ on its boundary. The inward normal is $-\partial_r$, so $\partial_{\nu}u=-U'(R)=q(R)=c$. This proves solvability.
Summing up, we complete the proof of Theorem \ref{re1}.
\section{Further consequences and remarks}
We give some further consequences and remarks in this final section. We first notice that the proof of Theorem 1.1 is directionwise. This observation gives a useful partial conclusion when reflection information is available only in one direction. 
\vskip 0.1in
\begin{corollary}\label{cor1}.\; 
Fix $e\in\mathbb{S}^{n-1}$. Assume that $f\geq 0$, $f\not\equiv 0$, and that (F)-(G) hold for the extremely admissible cap $(\sum_{e,\lambda_e}, T_{e,\lambda_e})$ of $\Omega$ along the direction $e$. Then $\Omega$ is symmetric about the plane $T_{e,\lambda_e}$. If the same hypotheses also hold in the opposite direction $-e$, then $u$, $f$ and $g$ are symmetric across the plane $T_{e,\lambda_e}$.
\end{corollary}
\vskip 0.05in
{\bf Proof:}\; The conclusion of this corollary follows immediately from the proof of Theorem 1.1. 
\vskip 0.1in
The preceding result can be amplified from one direction to a generating set of directions. For $e\in\mathbb{S}^{n-1}$, let $T_{e}=\{x\in\mathbb{R}^n:\;x\cdot e=0\}$, and $\mathcal{R}_e(x)=x-2(x\cdot e)e$ be the linear reflection through $T_e$. For set $E\subset\mathbb{S}^{n-1}$, we denote by $\mathcal{G}(E)$ the subgroup of $O(n)$ generated by $\{\mathcal{R}_e:\;e\in E\}$.
\vskip 0.1in
\begin{theorem}\label{cor2}.\; Let $E\subset\mathbb{S}^{n-1}$ be a set such that $e\in E$ implies $-e\in E$ and assume that $0\leq f\not\equiv 0$ on $\overline{\Omega}$. For every $e\in E$, suppose (F)-(G) hold for the extremely admissible cap $(\sum_{e,\lambda_e}, T_{e,\lambda_e})$ of $\Omega$ along the direction $e$. If $\mathcal{G}(E)$ is dense in the orthogonal group $O(n)$, then $\Omega$ is a ball, $u$ and $f$ are radial about the center of the ball and $g$ is a constant.
\end{theorem}
\vskip 0.05in
{\bf Proof:}\; For each $e\in E$, Corollary \ref{cor1} gives an affine reflection symmetry $\mathcal{R}_{e,\lambda_e}$ of $\Omega$. Let 
$$x_0=\frac{1}{|\Omega|}\int_{\Omega}xdx$$
be the geometric center of $\Omega$. Since every symmetry preserves the geometric center, $x_0\cdot e=\lambda_e$. After translating $x_0$ to the origin, $\Omega$ is therefore invariant under every $\mathcal{R}_e$ and hence under $\mathcal{G}(E)$. Evidently, invariance passes to the closure of this group because $Q_k\in\mathcal{G}(E)$, $Q_k\rightarrow Q$, and $B_{r}(x)\subset\Omega$ implies $B_{r/2}(Qx)\subset\Omega$ for large $k$ due to $B_r(Q_kx)\subset\Omega$. Thus density gives invariance of $\Omega$ under all of $O(n)$.

It is well known that a bounded rotationally invariant domain with $C^2$ boundary is either a ball centered at the origin or a concentric annulus since boundary regularity excludes a punctured ball. For any $a>0$, we can see that the annular $\{x\in\mathbb{R}^n:a<|x|<R\}$ is also impossible here. Indeed, for any $e\in E$, moving the plane $T_{e,\lambda}$ continuously inward to a position such that $\lambda=(R+a)/2=\lambda_0$, we have $\mathcal{R}_{e,\lambda_0}(Re)=ae$. Therefore,  
for any $\lambda<\lambda_0$ and close to $\lambda_0$, we have $\mathcal{R}_{e,\lambda}(Re)$ is reflected into the hole. This implies that the extreme admissible plane obtained by continuously moving a plane along direction $e$ would not pass through the origin, contradicting the conclusion obtained by moving plane method. Hence, $\Omega$ is a ball, $f$ is radial, and $g$ is a constant due to assumptions on $E$, $f$ and $g$. 
\vskip 0.1in
\noindent {\bf Remark 4.1}.\; The density assumption in Theorem \ref{cor2} cannot be removed.
\vskip 0.1in
\noindent {\bf Remark 4.2}.\; Nonnegativity of $f$ enters the main theorem only to establish positivity of the solution. It can therefore be replaced by positivity of the solution itself.
\vskip 0.1in
\begin{corollary}\label{cor3}.\;Let $f\in C^\alpha(\overline{\Omega})$ be allowed to change sign. Assume that (F)-(G) hold. Then the existence of a positive solution to problem \ref{prob4} forces $\Omega$ to be a ball, $f$ to be radial, and $g$ to be a suitable constant.
\end{corollary}

\noindent {\bf Remark 4.3}.\; Recently, Serrin's rigidity result was generated by Figalli and Zhang \cite{FZ} to rough domains. We also like to know if the $C^{2,\alpha}$ assumption of domain in our main theorem can be weakened.
\vskip 0.1in
\noindent {\bf Data availability:}\;Data sharing is not applicable to this article as no datasets were generated or analyzed during the current study.
\vskip 0.1in
\noindent {\bf Conflict of interest:}\;The author declare that there is no conflict of interest in this manuscript.


\begin{thebibliography}{s2}

\bibitem{Al} A. D. Alexandrov, Uniqueness theorems for surfaces in the large, V.(Russian) Vestnik Leningrad. Univ. 13, 5-8, 1958

\bibitem{Ale} A. D. Alexandrov, A characteristic property of the spheres, Ann. Mat. Pura Appl. 58, 303-354, 1962

\bibitem{BN} H. Berestycki, L. Nirenberger, On the method of moving planes and sliding method, Bol. Soc. Brasil. Math. (N.S.) 22, 1-37, 1991

\bibitem{BHS} C. Brandolini, A. Henrot, P. Salani, An overdetermined problem with non-constant boundary condition, Int. Free. Bound., 16(2), 215-242, 2014

\bibitem{BH} F. Brock, A. Henrot, A symmetry result for an overdetermined elliptic problem using continuous rearrangement and domain derivative, Rend. Circ. Mat. Palermo 51, 375-390, 2002

\bibitem{CWZ} D. Cao, J. Wei, W. Zhan, Symmetry in Serrin-type overdetermined problems, arXiv preprint arXiv:2506.02423, 2025

\bibitem{Cr} G. Crasta, A symmetry problem in the calculus of variations, J. Europ. Math. Soc. 8, 139-154, 2006

\bibitem{DP} L. Damascelli, F. Pacella, Monotonicity and symmetry results for p-Laplace equations and applications, Adv. Differential Equation 5, 1179-1200, 2000

\bibitem{FZ} A. Figalli, R.-Y. Y. Zhang, Serrin's overdetermined problem in rough domains, arXiv:2407.02293

\bibitem{FGK} I. Fragala, F. Gazzola, B. Kawohl, Overdetermined problems with possibly degenerate ellipticity, a geometric approach, Math. Z. 254, 117-132, 2025

\bibitem{GL} N. Garofalo, J. L. Lewis, A symmetry result related to some overdetermined boundary value problems, Amer. J. Math. 111, 9-33, 1989

\bibitem{GS} N. Garofalo, E. Sartori, Symmetry in exterior boundary value problems for quasilinear elliptic equations via blow-up and a priori estimates, Adv. Diff. Eq. 4, 137-161, 1999


\bibitem{GO} A. Gilsbach, M. Onodera, Linear stability estimates for Serrin's problem via a modified implicit function theorem, Calc. Var. PDE., 60:241, 2021


\bibitem{KP} S. Kumaresan, J. Prajapat, Serrin's result for hyperbolic space and sphere, Duke Math. J. 91, 17-28, 1998

\bibitem{MP} R. Magnanini, G. Poggesi, Nearly optimal stability for Serrin's problem and the soap bubble theorem, Calc. Var. PDE., 59:35, 2020

\bibitem{Pr} J. Prajapat, Serrin's result for domains with a corner or cusp, Duke Math. J. 91, 29-31, 1998

\bibitem{PS} P. Pucci, J. Serrin, The strong maximum principle revisited, J. Differential Equations., 196(1):1-66, 2004

\bibitem{SC} S. Sakaguch, Concavity property of solutions to some degenerate quasilinear elliptic Dirichlet problems, Annali della Scuola Normale Superiore de Pisa, Serie IV(Classe di Scienze) 14, 403-421, 1987

\bibitem{Se} J. Serrin, A symmetry problem in potential theory, Arch. Ration. Mech. Anal. 43, 304-318, 1971

\bibitem{TR} P. Tolksdorff, Regularity for a more general class of quasilinear elliptic equations, J. Differential Equation., 51, 126-150, 1984

\bibitem{We} H. F. Weinberger, Remark on the preceding paper of Serrin, Arch. Ration. Mech. Anal., 43(4), 319-320, 1971


\end{thebibliography}
\end{document}